\documentclass[11pt]{amsart}
\usepackage{amssymb,amsmath,amsthm,newlfont,enumerate}

\theoremstyle{plain}
\newtheorem{theorem}{Theorem}[section]
\newtheorem{proposition}[theorem]{Proposition}

\theoremstyle{definition}

\theoremstyle{remark}

\newtheorem*{remark}{Remark}

\newcommand{\BB}{{\mathbb B}}
\newcommand{\CC}{{\mathbb C}}
\newcommand{\DD}{{\mathbb D}}

\newcommand{\TT}{{\mathbb T}}

\newcommand{\cD}{{\mathcal D}}

\newcommand{\cH}{{\mathcal H}}
\newcommand{\cM}{{\mathcal M}}

\begin{document}

\date{20 Aug 2021}

\title[A GKZ theorem for reproducing kernel Hilbert spaces]{A Gleason--Kahane--\.Zelazko theorem for reproducing kernel Hilbert spaces}

\author[Chu]{Cheng Chu}
\address{D\'epartement de math\'ematiques et de statistique, Universit\'e Laval,
Qu\'ebec City (Qu\'ebec),  Canada G1V 0A6.}
\email{chengchu813@gmail.com}

\author[Hartz]{Michael Hartz}
\address{Fachrichtung Mathematik, Universit\"at des Saarlandes, 66123 Saarbr\"ucken, Germany}
\email{hartz@math.uni-sb.de}

\author[Mashreghi]{Javad Mashreghi}
\address{D\'epartement de math\'ematiques et de statistique, Universit\'e Laval,
Qu\'ebec City (Qu\'ebec),  Canada G1V 0A6.}
\email{javad.mashreghi@mat.ulaval.ca}

\author[Ransford]{Thomas Ransford}
\address{D\'epartement de math\'ematiques et de statistique, Universit\'e Laval,
Qu\'ebec City (Qu\'ebec),  Canada G1V 0A6.}
\email{thomas.ransford@mat.ulaval.ca}

\begin{abstract}
We establish the following Hilbert-space analogue of the Gleason--Kahane-\.Zelazko theorem.
If $\cH$ is a reproducing kernel Hilbert space with a normalized complete Pick kernel,
and if $\Lambda$ is a linear functional on $\cH$ such that $\Lambda(1)=1$ and $\Lambda(f)\ne0$
for all cyclic functions $f\in\cH$, then $\Lambda$ is multiplicative,
 in the sense that $\Lambda(fg)=\Lambda(f)\Lambda(g)$
for all $f,g\in\cH$ such that $fg\in\cH$. 
Moreover $\Lambda$ is automatically continuous.
We give examples to show that the theorem fails if the hypothesis of a complete Pick kernel is omitted.
We also discuss conditions under which $\Lambda$ has to be a point evaluation.
\end{abstract}



\maketitle


\section{Introduction and statement of main result}\label{S:intro}

The following result, known as the Gleason--Kahane--\.Zelazko (GKZ) theorem,
char\-acter\-izes multiplicativity of a linear functional on a Banach algebra.

\begin{theorem}[\cite{Gl67,KZ68,Ze68}]\label{T:GKZ}
Let $A$ be a complex unital Banach algebra, 
and let $\Lambda:A\to\CC$ be a linear functional such that $\Lambda\not\equiv0$.
The following statements are equivalent:
\begin{enumerate}[\normalfont\rm(i)]
\item $\Lambda(1)=1$ and $\Lambda(a)\ne0$ for all invertible elements $a\in A$. 
\item $\Lambda(ab)=\Lambda(a)\Lambda(b)$ for all $a,b\in A$.
\end{enumerate}
\end{theorem}

Continuity of the functional is not assumed,
but it is  a consequence of the theorem, since
characters on a Banach algebra are always continuous.

Analogues of this result for certain holomorphic function spaces 
that are not algebras were obtained in \cite{MRR18,MR15}.
Our goal in this paper is to establish the following theorem, which is a version for abstract function spaces.

\begin{theorem}\label{T:main}
Let $\cH$ be a reproducing kernel Hilbert space with a normalized complete Pick kernel,
and let $\Lambda:\cH\to\CC$ be a linear functional such that $\Lambda\not\equiv0$.
The following statements are equivalent:
\begin{enumerate}[\normalfont\rm(i)]
\item $\Lambda(1)=1$ and $\Lambda(f)\ne0$ for all cyclic elements $f\in \cH$. 
\item $\Lambda(fg)=\Lambda(f)\Lambda(g)$ for all $f,g\in\cH$ such that $fg\in\cH$.
\end{enumerate}
Moreover, a functional $\Lambda$ satisfying (i) or (ii) is automatically continuous.
\end{theorem}

The terminology is explained  in \S\ref{S:background},
the automatic continuity of $\Lambda$ is established in \S\ref{S:autocty},
and the rest of the theorem proved in \S\ref{S:proof}.
In \S\ref{S:example} we present  examples showing that the theorem may fail 
if we drop the assumption that $\cH$ has a complete Pick kernel.
Finally, in \S\ref{S:ptevals} we discuss conditions under which we may conclude that
$\Lambda$ is a point evaluation at a point of $X$. 


\section{Background on reproducing kernel Hilbert spaces}\label{S:background}

In this section we give a very brief introduction to reproducing kernel Hilbert spaces
and complete Pick kernels. For more details we refer to the books of 
Paulsen--Raghupathi \cite{PR16} and Agler--McCarthy \cite{AM02}.

\subsection{Reproducing kernel Hilbert spaces}
A \emph{reproducing kernel Hil\-bert space} (RKHS for short)
is a complex Hilbert space $\cH$  of functions on a set $X$ with the property that,
for each $x\in X$, the evaluation functional $f\mapsto f(x)$ is continuous  on $\cH$.
By the Riesz representation theorem, there is a uniquely determined element $k_x\in\cH$
such that $f(x)=\langle f,k_x\rangle$ for all $f\in\cH$.
The \emph{reproducing kernel} associated to $\cH$
is the function $K: X\times X\to\CC$ defined by $K(x,y):=k_y(x)=\langle k_y,k_x\rangle$.

A kernel $K$ is automatically a positive semi-definite function.
This means that, for every finite subset $\{x_1,x_2,\dots,x_n\}$ of $X$,
the $n\times n$ matrix $(K(x_i,x_j))$ is positive semi-definite.

We shall always assume that $K(x,x)>0$ for all $x\in X$.
This is equivalent to supposing that, for each  $x\in X$, there exists
$f\in\cH$ such that $f(x)\ne0$. In particular, this is the case if $K$ is normalized 
(see \S\ref{S:normalized} below).

\subsection{Multipliers}\label{S:multipliers}
Let $\cH$ be a RKHS on $X$.
A \emph{multiplier} of $\cH$ is a function $h:X\to\CC$ such that $hf\in\cH$ whenever $f\in\cH$.
We denote by $\cM$ the set of multipliers.
By the closed graph theorem, if $h\in\cM$, then $M_h:f\mapsto hf$ is a bounded linear operator on $\cH$.
We define $\|h\|_\cM$ to be the operator norm of $M_h$. With this norm $\cM$ becomes a Banach algebra.

If $h\in\cM$ and  $x\in X$, then
a standard calculation shows that $M_h^*k_x=\overline{h(x)}k_x$, so $\overline{h(x)}$
is an eigenvalue of $M_h^*$, and $|h(x)|\le\|h\|_\cM$.
Consequently, every multiplier $h$ is a bounded function on $X$ with $\sup_X|h|\le\|h\|_\cM$.

\subsection{Complete Pick kernels}
Let $K$ be the reproducing kernel of a RKHS $\cH$ on $X$.
It is a \emph{complete Pick kernel}  if $K(x,y)\ne0$ for all $x,y\in X$ and if
there exists $x_0\in X$ such that
\begin{equation}
  \label{eqn:CNP}
  F(x,y):=1-\frac{K(x,x_0)K(x_0,y)}{K(x,y)K(x_0,x_0)}
\end{equation}
is a positive semi-definite function on $X\times X$.
If this condition holds for one $x_0\in X$, then it automatically holds for all $x_0\in X$.

Complete Pick kernels are so called because of their relation with the solution of
Pick interpolation problems. This is described in detail in \cite{AM02}.
Examples of RKHS with complete Pick kernels include the Hardy space~$H^2$,
the Drury--Arveson spaces $H^2_d$ \cite{AM02}, the classical Dirichlet space $\cD$ \cite{Ag88}, 
all Dirichlet spaces with superharmonic weights~$\cD_w$ \cite{Sh02}, 
the Sobolev space~$W^2_1(0,1)$ \cite{Qu93},
as well as certain de Branges--Rovnyak spaces~$\cH(b)$ \cite{Ch20}.

\subsection{Normalized kernels}\label{S:normalized}
A reproducing kernel $K$ for a RKHS $\cH$ on $X$
is said to be \emph{normalized} if there exists  $x_0\in X$ such that $K(x,x_0)=1$ for all $x\in X$.
In this case $k_{x_0}(x):=K(x,x_0)=1$ for all $x\in X$. Since $k_{x_0}\in\cH$,
it follows that $\cH$ contains the constant functions, and hence that $\cM\subset\cH$.

We require the following well-known fact, which follows for instance from the discussion
surrounding Equation (2.2) in \cite{GRS02}. For the convenience of the reader, we provide a short proof.

\begin{proposition}
  \label{P:kernels_multipliers}
  Let $\cH$ be a RKHS on $X$ with a normalized complete Pick kernel $K$ and let $\cM$ be its multiplier
  algebra. Then $k_y = K(\cdot,y) \in \cM$ for all $y \in X$. In particular,
  $\cM$ is dense in $\cH$.
\end{proposition}

\begin{proof}
  Since $K$ is a normalized complete Pick kernel, a theorem of Agler and McCarthy \cite[Theorem 4.2]{AM00} shows
  that there exists an auxiliary Hilbert
  space $\mathcal{E}$ and a function $b: X \to B(\mathcal{E},\CC)$
  with $\|b(x)\| < 1$ for all $x \in X$ such that
  \begin{equation*}
    K(x,y) = \frac{1}{1 - b(x) b(y)^*} \quad (x,y \in X).
  \end{equation*}
  Indeed, this follows by choosing $x_0$ in Equation \eqref{eqn:CNP} to be the normalization point,
  defining $\mathcal{E}$ to be the RKHS with reproducing kernel $F$ and $b(x)$ to be the functional on $\mathcal{E}$
  of evaluation at $x$.

  Since
  \begin{equation*}
    (x,y) \mapsto (1 - b(x) b(y)^*) K(x,y) = 1
  \end{equation*}
  is a positive semi-definite function on $X \times X$, the function $b$ is a multiplier
  from $\cH \otimes \mathcal{E}$ into $\cH$ of norm at most $1$
  (see e.g.\ \cite[Theorem~6.28]{PR16}).
  From this, we deduce that for each $y \in X$, the scalar-valued function
  $x \mapsto b(x) b(y)^*$ belongs to $\cM$ and
  \begin{equation*}
    \| b(\cdot) b(y)^*\|_{\cM} \le \|b(y)^*\|_{\mathcal{E}} < 1.
  \end{equation*}
  Therefore, the series
  \begin{equation*}
    K(\cdot,y) = \sum_{n=0}^\infty (b(\cdot) b(y)^*)^n
  \end{equation*}
  converges absolutely in the Banach algebra $\cM$. Hence $K(\cdot,y) \in \cM$ 
  for all $y \in X$. Since the linear span of the kernel functions is dense in any RKHS,
  it follows in particular that $\cM$ is dense in $\cH$.
\end{proof}

\subsection{Cyclic functions}
\label{SS:cyclic}
Let $\cH$ be a RKHS on $X$ with multiplier algebra~$\cM$.
Given $f\in\cH$, we write $[f]:=\overline{\cM f}$, the closed $\cM$-invariant subspace generated by $f$.
Note that, if $h\in\cM$, then $h[f]\subset[hf]$.
We say $f$ is \emph{cyclic} if $[f]=\cH$. 
Clearly, if $f \in \cH$ is cyclic, then $f(x)\ne0$ for all $x\in X$.
In general, the reverse is not true.

It should be remarked that some authors (notably those in \cite{AHMR17}) 
define a multiplier $h \in \cM$ to be `cyclic' if $\overline{h\cH} = \cH$. 
If $\cM$ is dense in $\cH$, then the two notions of cyclic coincide for all $h\in \cM$.
In particular, this is the case whenever $\cH$ has a normalized
complete Pick kernel, by Proposition~\ref{P:kernels_multipliers}.

Cyclic functions play a role in RKHS analogous to that of invertible elements in algebras.
In the case of the Hardy space $H^2$, a function is cyclic if and only if it is an outer function.
This is a consequence of Beurling's classification of the closed, shift-invariant subspaces of $H^2$.
In many other spaces, it is difficult to characterize which functions are cyclic.
In particular, in the case of the classical Dirichlet space, this problem is the subject of 
a well-known conjecture of Brown and Shields \cite{BS84}.


\section{Automatic continuity}\label{S:autocty}

As remarked in the introduction, it is well known (and easy to prove) that a multiplicative 
linear functional on a unital Banach algebra  is automatically continuous. The corresponding 
result for a RKHS with a normalized, complete Pick kernel is also true, but not quite so obvious.
In fact, it is even true under the weaker multiplicativity hypothesis that
$\Lambda(hf)=\Lambda(h)\Lambda(f)$ for all $f\in\cH$ and all $h\in\cM$.
As it will actually play a role in our proof of Theorem~\ref{T:main}, we stop here to prove it  first.

\begin{theorem}\label{T:autocty}
Let $\cH$ be a RKHS with a normalized complete Pick kernel,
let $\cM$ be its multiplier algebra, and let $\Lambda:\cH\to\CC$
be a linear functional such that 
$\Lambda(hf)=\Lambda(h)\Lambda(f)$ for all $f\in\cH$ and all $h\in\cM$.
Then $\Lambda$ is continuous on $\cH$.
\end{theorem}

The proof of Theorem~\ref{T:autocty} hinges on the
following factorization theorem due to Jury and Martin \cite[Theorem~1.1]{JM19}.
This result was proved by using non-commutative factorization theorems
due to Arias and Popescu \cite{AP95} and Davidson and Pitts \cite{DP99}.

\begin{theorem}\label{T:JM}
Let $\cH$ be a RKHS with a normalized complete Pick kernel, and let $\cM$ be its multiplier algebra.
If $(f_n)$ is a sequence in $\cH$ such that $\sum_n\|f_n\|_\cH^2<\infty$, then there exists
a sequence $(h_n)$ in $\cM$ and a cyclic function $g\in\cH$ such that:
\begin{itemize}
\item $f_n=h_ng$ for all $n$;
\item $\sum_n\|f_n\|_{\cH}^2= \|g\|_{\cH}^2$;
\item $\sum_n\|h_nf\|_{\cH}^2\le \|f\|_{\cH}^2$ for all $f\in\cH$, and in particular $\|h_n\|_\cM\le 1$
for all $n$.
\end{itemize}
\end{theorem}

\begin{proof}[Proof of Theorem~\ref{T:autocty}]
Suppose, if possible, that $\Lambda$ is not continuous on $\cH$.
Then there exists a sequence $(f_n)$ in $\cH$ such that $\sum_n\|f_n\|_{\cH}^2<\infty$
but $|\Lambda(f_n)|\to\infty$. By Theorem~\ref{T:JM}, we can factorize $f_n$ as $f_n=h_ng$,
where $g\in\cH$ and $(h_n)$ is a sequence in the multiplier algebra $\cM$ of $\cH$
such that $\|h_n\|_\cM\le 1$ for all $n$. By the multiplicativity hypothesis on $\Lambda$,
we have $\Lambda(f_n)=\Lambda(h_n)\Lambda(g)$ for all $n$, which forces $|\Lambda(h_n)|\to\infty$.
On the other hand, $\Lambda|_\cM$ is a character on the Banach algebra $\cM$, so it is automatically continuous
on $\cM$,
and as the sequence $(h_n)$ is bounded in the norm of $\cM$, it follows that $\Lambda(h_n)$ is bounded too.
We have arrived at a contradiction.
\end{proof}


\section{Proof of Theorem~\ref{T:main}}\label{S:proof}

The proof of Theorem~\ref{T:main} is based on the following  result,
which is a version of the  classical GKZ theorem for modules
\cite[Theorem~1.2]{MR15}.

\begin{theorem}\label{T:GKZmodule}
Let $A$ be a  complex unital Banach algebra, let $M$ be a left $A$-module,
and let $S$ be a non-empty subset of $M$ satisfying the following conditions:
\begin{itemize}
\item[(S1)] $S$ generates $M$ as an $A$-module;
\item[(S2)] if $a\in A$ is invertible and $s\in S$, then $a s\in S$;
\item[(S3)] for all $s_1,s_2\in S$, there exist $a_1,a_2\in A$ such that 
$a_j S\subset S~(j=1,2)$ and $a_1 s_1=a_2 s_2$.
\end{itemize}
Let $\Lambda:M\to\CC$ be a linear functional such that 
$\Lambda(s)\ne0$ for all $s\in S$. 
Then there exists a unique character $\chi$ on $A$ such that
\begin{equation}\label{E:lambdachi}
\Lambda(a m)=\chi(a)\Lambda(m) \qquad(a\in A,~m\in M).
\end{equation}
\end{theorem}

We also need another factorization theorem for RKHS's, this one due to
Aleman, McCarthy, Richter  and the second author \cite[Lemma~2.3 and Theorem~3.1]{AHMR17}.
See also  \S\ref{SS:cyclic} regarding the definition of cyclicity used in \cite{AHMR17}.
 
\begin{theorem}\label{T:AHMR}
Let $\cH$ be a RKHS on $X$ with a normalized complete Pick kernel, and let $\cM$ be its multiplier algebra.
Then, given $f\in\cH$, there exist $h_1,h_2\in\cM$, with $h_2$ cyclic in $\cH$, such that $f=h_1/h_2$.
\end{theorem}

\begin{proof}[Proof of Theorem~\ref{T:main}]
We begin with the implication (i)$\Rightarrow$(ii).
Assume that (i) holds, namely $\Lambda(1)=1$ and $\Lambda(f)\ne0$ for all cyclic functions $f\in\cH$.

The plan is to apply Theorem~\ref{T:GKZmodule}
with $M=\cH$, taking $A$ to be the multiplier algebra $\cM$ of $\cH$, and
$S$ to be  the set of cyclic functions in $\cH$.
We verify the three conditions in that theorem:
\begin{itemize}
\item[(S1)] Given $f\in\cH$,  Theorem~\ref{T:JM} implies that we can write $f=hg$, where $h\in\cM$ and $g\in S$.
\item[(S2)] If $h$ is invertible in $\cM$, then $\cM(hf)=\cM f$. In particular, if $f$ is cyclic,  so is $hf$.
\item[(S3)] Given  $f_1,f_2\in S$, Theorem~\ref{T:JM} shows that we can write $f_1=a_2g$ and $f_2=a_1g$,
where $a_1,a_2\in\cM$ and $g\in S$.
Clearly we have $a_1f_1=a_2f_2$.
It remains to check that $a_jS\subset S$.
Let $f\in S$. Then $[a_1f]\supset a_1[f]=a_1\cH\ni a_1g=f_2$,
and so $[a_1f]\supset[f_2]=\cH$. Hence $a_1f\in S$. Likewise for $a_2$.
\end{itemize}

By Theorem~\ref{T:GKZmodule}, there exists a character $\chi$ on $\cM$ such that \eqref{E:lambdachi}
holds, i.e.,
\[
\Lambda(hf)=\chi(h)\Lambda(f) \quad (h\in\cM,~f\in\cH).
\]
In particular, taking $f=1$, we see that $\chi(h)=\Lambda(h)$ for all $h\in\cM$. It follows that 
\begin{equation}\label{E:weakconcl}
\Lambda(hf)=\Lambda(h)\Lambda(f) \quad (h\in\cM,~f\in\cH).
\end{equation}

This is nearly the desired conclusion (ii).
To get the full conclusion, we employ a trick borrowed from \cite{AHMR17}.
Let $f,g\in\cH$ be such that $fg\in\cH$. 
By Theorem~\ref{T:AHMR}, we can write $f=h_1/h_2$, 
where $h_1,h_2\in\cM$ and $h_2$ is cyclic in $\cH$.
Then, using \eqref{E:weakconcl} several times, we have
\begin{align*}
\Lambda(h_2)\Lambda(fg)&=\Lambda(h_2fg)=\Lambda(h_1g)=\Lambda(h_1)\Lambda(g)\\
&=\Lambda(h_2f)\Lambda(g)=\Lambda(h_2)\Lambda(f)\Lambda(g).
\end{align*}
Since $h_2$ is cyclic, we have $\Lambda(h_2)\ne0$, so we can cancel off the $\Lambda(h_2)$ terms
and conclude that $\Lambda(fg)=\Lambda(f)\Lambda(g)$, as desired.
This completes the proof that (i)$\Rightarrow$(ii).

We now turn to the reverse implication in Theorem~\ref{T:main}, namely (ii)$\Rightarrow$(i).
Assume that (ii) holds, i.e., $\Lambda(fg)=\Lambda(f)\Lambda(g)$ for all $f,g\in\cH$
with  $fg\in\cH$.

Since $\Lambda(1)=\Lambda(1\cdot 1)=\Lambda(1)^2$, we have $\Lambda(1)=0$ or $1$.
In the former case, we have $\Lambda(f)=\Lambda(1\cdot f)=\Lambda(1)\Lambda(f)=0$ for all $f\in\cH$,
contrary to the hypothesis that $\Lambda\not\equiv0$. So $\Lambda(1)=1$.

Let $f$ be a cyclic function in $\cH$. From the definition of cyclicity, there exists a sequence
of multipliers $(h_n)$ such that $h_nf\to1$ in $\cH$. 
By Theorem~\ref{T:autocty}, the assumption that $\Lambda$ is multiplicative implies that $\Lambda$
is continuous on $\cH$, and so $\Lambda(h_n)\Lambda(f)=\Lambda(h_nf)\to\Lambda(1)=1$.
This implies that $\Lambda(f)\ne0$, and
completes the proof that (ii)$\Rightarrow$(i).

Finally, the continuity statement in Theorem~\ref{T:main} follows directly from Theorem~\ref{T:autocty}.
\end{proof}


\section{Two examples on the complete Pick kernel condition}\label{S:example}

In this section we present examples showing that Theorems~\ref{T:main} and~\ref{T:autocty}
are no longer true if we drop the assumption that $\cH$
has a complete Pick kernel. We begin with an example relating to Theorem~\ref{T:main}.

\begin{proposition}
  \label{P:example}
There exist a RKHS $\cH$ on a set $X$ and a linear functional $\Lambda:\cH\to\CC$ with the following properties:
\begin{enumerate}[\normalfont\rm(i)]
\item The reproducing kernel $K:X\times X\to\CC$ associated to $\cH$ is normalized, and satisfies $K(x,y)\ne0$ for all $x,y\in X$.
\item The linear functional $\Lambda$ satisfies $\Lambda(1)=1$ and $\Lambda(f)\ne0$ for all cyclic elements $f\in\cH$.
Furthermore $\Lambda$ is continuous on $\cH$.
\item There exists a multiplier $h$ of $\cH$ such that $\Lambda(h^2)\ne \Lambda(h)^2$.
\end{enumerate}
\end{proposition}

\begin{proof}
(i) Let 
\[
\cH:=\Bigl\{ f(z_1,z_2):=\sum_{j,k\ge0}a_{jk}z_1^jz_2^k:~ \|f\|_\cH^2:=\sum_{j,k\ge0}j!|a_{jk}|^2<\infty\Bigr\}.
\]
This is a RKHS of holomorphic functions on $\CC\times\DD$. 
It has the orthonormal basis $\{z_1^jz_2^k/(j!)^{1/2}:j,k\ge0\}$,
from which we deduce that its reproducing kernel is given by
\[
K\bigl((z_1,z_2),(w_1,w_2)\bigr)
=\sum_{j,k\ge0}\frac{z_1^jz_2^k}{(j!)^{1/2}}\frac{\overline{w}_1^j\overline{w}_2^k}{(j!)^{1/2}}
=\frac{e^{z_1\overline{w}_1}}{1-z_2\overline{w}_2}.
\]
Clearly $K$ is normalized at $(0,0)$, and it is non-zero everywhere on $\CC\times\DD$.
(In fact $\cH$ is the  tensor product of the Segal--Bargmann space and the Hardy space, 
and so $K$ is the product of the kernels of these two spaces.)

(ii) Define $\Lambda:\cH\to\CC$ by
\[
\Lambda\Bigl(\sum_{j,k\ge0}a_{jk}z_1^jz_2^k\Bigr):=a_{00}+a_{01}.
\]
This is a continuous linear functional on $\cH$. Clearly $\Lambda(1)=1$.
Also, the condition that $\Lambda(f)\ne0$ for all cyclic $f\in\cH$ is satisfied vacuously, 
because there are no cyclic elements of $\cH$.
Indeed, every multiplier of $\cH$ is a bounded holomorphic function on $\CC\times\DD$,
so by Liouville's theorem it must be independent of the first variable. This evidently precludes
the existence of cyclic functions.

(iii) Let $h(z_1,z_2):=z_2$ . 
Multiplication by $h$ is an isometry on $\cH$, so $h$ is certainly a multiplier. Also, we have
$\Lambda(h)=1$
and $\Lambda(h^2)=0$,
so $\Lambda(h^2)\ne\Lambda(h)^2$.
\end{proof}

\begin{remark}
In the example constructed in the proof of Proposition~\ref{P:example},
the multiplier algebra $\cM$ is not dense in $\cH$.
This is no accident.
Indeed, if $\cM$ is dense in $\cH$, then
every invertible multiplier
is a cyclic element of $\cH$. Thus the classical GKZ theorem (Theorem~\ref{T:GKZ}),
applied to the restriction of $\Lambda$ to $\cM$, shows that
\[
\Lambda(h_1 h_2) = \Lambda(h_1) \Lambda(h_2) \quad (h_1,h_2 \in \cM),
\]
and if, further, $\Lambda$ is continuous on $\cH$, then it follows that
\[
 \Lambda(h f) = \Lambda(h) \Lambda(f) \quad (h \in \cM, f \in \cH).
\]
\end{remark}

Concerning the continuity of $\Lambda$,
one might also ask if the automatic continuity result, Theorem \ref{T:autocty}, holds 
if we drop the assumption that $\cH$ has a complete Pick kernel.
The following example shows that this is not the case,
even if the multiplier algebra $\cM$ is dense in $\cH$.

\begin{proposition}
There exist a RKHS $\cH$ of holomorphic functions on $\DD$ and a linear functional
$\Lambda: \cH \to \CC$ with the following properties:
\begin{enumerate}[\normalfont\rm(i)]
    \item The reproducing kernel $K$ associated to $\cH$ is normalized at $0$.
    \item The multiplier algebra $\cM$ of $\cH$ equals $H^\infty$ and is dense in $\cH$.
    \item The linear functional $\Lambda$ satisfies $\Lambda(1) = 1$ and $\Lambda(h f) = \Lambda(h) \Lambda(f)$
      for all $h \in \cM$ and all $f \in \cH$.
    \item $\Lambda$ is discontinuous on $\cH$.
  \end{enumerate}
\end{proposition}

\begin{proof}
Let $(a_n)_{n=0}^\infty$ be a sequence of strictly positive real numbers with the following properties:
 \begin{enumerate}[(a)]
    \item $a_0 = 1$,
    \item $a_n/a_{n+1} \le 1$ for all $n\ge0$,
    \item $\inf_{n\ge0} a_n/a_{n+1} = 0$, and
    \item $\lim_{n \to \infty} a_n^{1/n} = 1$.
  \end{enumerate}
For instance, we can define $a_n:= k!$ whenever $k^2 \le n < (k+1)^2$.
Then properties~(a) and (b) are obviously satisfied. Moreover, (c) holds since
\[
    \frac{a_{(k+1)^2-1}}{a_{(k+1)^2}} = \frac{k!}{(k+1)!} = \frac{1}{k+1} \xrightarrow{k \to \infty} 0.
\]
Finally, to see (d), notice that, if $k^2 \le n < (k+1)^2$, then
\[
    1 \le a_n^{1/n} \le (k!)^{1/k^2} \le k^{1/k} \xrightarrow{k \to \infty} 1.
\]

  Next, we define a positive semi-definite kernel $K$ on $\DD$ by
\[
    K(z,w) := \sum_{n=0}^\infty a_n (z \overline{w})^n \quad (z,w \in \DD).
  \end{equation*}
Condition~(d) ensures that the power series defining $K$ has radius of convergence $1$.
  Let $\cH$ be the RKHS on $\DD$ with reproducing kernel $K$.
  It is well known and not hard to see that $\cH$ consists of holomorphic functions on $\DD$, and that the monomials
  $(z^n)_{n=0}^\infty$ form an orthogonal basis of $\cH$ with
  \begin{equation}
    \label{eqn:norm_monomials}
    \|z^n\|^2 = \frac{1}{a_n} \quad (n \ge0),
  \end{equation}
  see e.g.\ \cite[\S4.4.2]{PR16}.
  Since $a_0 = 1$, the reproducing kernel $K$ is normalized at $0$.

We next show  that the multiplier algebra $\cM$ of $\cH$ equals $H^\infty$.
Certainly $\cM$ is contractively contained in $H^\infty$,
as was already remarked in \S\ref{S:multipliers}.
On the other hand, 
condition~(b) above can be interpreted as saying  that $a_n/a_{n+1}\le b_n/b_{n+1}$ for all $n\ge0$,
where $b_n := 1$ for all $n\ge0$,  namely the power series coefficients of the reproducing kernel of $H^2$.
In this setting, \cite[Proposition 3.3]{AHMR19} implies that
$H^\infty$, the multiplier algebra of $H^2$, is contractively contained in $\cM$.
Hence equality holds. 
In particular, we see that $\cM$ is dense in $\cH$.

To construct the functional $\Lambda$, we first
show that the (non-closed) subspace $z \cH$ has infinite codimension
in $\cH$. To this end, notice that $M_z$ is an injective operator on $\cH$,
and condition~(c) and equation~\eqref{eqn:norm_monomials}
together imply that it is not bounded below.
Therefore the range of $M_z$ is not closed and thus it has infinite codimension in $\cH$.
Hence there exists a discontinuous linear functional $\Lambda: \cH \to \CC$
such that
\[
    \Lambda(\alpha + z f) = \alpha \quad (\alpha \in \CC, ~f \in \cH).
\]
  
We finish the proof by showing the multiplicativity statement in (iii).
To this end, let $h \in \cM = H^\infty$ and let $f \in \cH$.
Then there exists $g \in H^\infty$ such that $h - h(0) = z g$.
In particular, $\Lambda(h) = h(0)$. Moreover, $h f = h(0) f +  z g f$ and so
\[
  \Lambda( h f) = \Lambda(h(0) f) + \Lambda( z g f) = h(0) \Lambda(f)= \Lambda(h) \Lambda(f),
\]
as claimed.
\end{proof}


\section{Point evaluations}\label{S:ptevals}

Let $\cH$ be a RKHS on $X$.
If $\Lambda:\cH\to\CC$ is a point evaluation at a point $x_0$ of $X$, 
namely $\Lambda(f):=f(x_0)$ for all $f\in\cH$,
then $\Lambda$ satisfies the hypothesis~(i) of Theorem~\ref{T:main}.  
Indeed, as remarked earlier, if $f\in\cH$ is cyclic, then $f(x)\ne0$ for all $x\in X$, 
and so $\Lambda(f)\ne0$. Obviously $\Lambda(1)=1$ as well.

This raises the question of whether every linear functional $\Lambda$ 
satisfying the hypotheses of Theorem~\ref{T:main}
is a point evaluation at some point of $X$. 
The answer is affirmative in certain  cases, 
notably the Hardy space $H^2$ on the unit disk \cite[Theorem~2.1]{MR15} and,
more generally, the Drury--Arveson space $H^2_d$ on the unit ball in $\CC^d$ for $1\le d\le\infty$ 
(combine Theorem~\ref{T:main} above with \cite[Lemma~5.3]{Ha17}).
Under the added assumption that $\Lambda$ is continuous,
further results of this kind have been established for various Banach function spaces on
the unit disk \cite[Theorem~3.1]{MR15} and on various domains in $\CC^n$ \cite[\S5]{Sa20}.

However, it is unrealistic to hope for an affirmative answer in the general 
situation considered in Theorem~\ref{T:main}. 
For example, if we think of $H^2$ as a RKHS on the punctured disk $\DD\setminus\{x_0\}$
for some fixed point $x_0 \in \DD \setminus \{0\}$,
then it still has a normalized complete Pick kernel,
and the linear functional $\Lambda(f):=f(x_0)$  is non-zero on cyclic functions,
but is not given by point evaluation at any point of $\DD\setminus\{x_0\}$.

A more `natural' example is the case when $\cH=\cD_\zeta$, 
the local Dirichlet space at a point $\zeta\in\TT$  \cite{RS91}.
This  is a RKHS on the  unit disk $\DD$,
and by a result of Shimorin \cite{Sh02} it has a complete Pick kernel. 
In this case, the functional $\Lambda(f):=\lim_{z\to\zeta}f(z)$ is well-defined and continuous on $\cD_\zeta$,
and is non-zero on cyclic functions, but it is not given by evaluation at any point of~$\DD$.

These examples suggest that, if we want to conclude that $\Lambda$ is a point evaluation,
then we need an assumption that ties $\Lambda$ more closely to the set $X$.
One possibility is to suppose that $\Lambda(f)\ne0$ for every $f\in\cH$ that is 
\emph{nowhere zero} on $X$. This eliminates the two counterexamples just mentioned.
For $H^2$ on $\DD\setminus\{x_0\}$ this is easy to see directly, 
and for $\cD_\zeta$ it is a special case of \cite[Theorem~1.3]{MRR18}.
However, as our final result shows,
in the abstract setting of this paper,
even this stronger hypothesis is not enough the yield the desired conclusion.

\begin{proposition}
There exist a RKHS $\cH$ on a set $X$ and a linear functional $\Lambda:\cH\to\CC$ with the following properties:
\begin{enumerate}[\normalfont\rm(i)]
\item The reproducing kernel $K:X\times X\to\CC$ associated to $\cH$ is a normalized complete Pick kernel.
\item The linear functional $\Lambda$ satisfies $\Lambda(1)=1$ and $\Lambda(f)\ne0$ for each $f\in\cH$ that is nowhere zero on $X$.
Furthermore $\Lambda$ is continuous on $\cH$.
\item $\Lambda$ is not given by point evaluation at any point of $X$.
\end{enumerate}
\end{proposition}

\begin{proof}
Fix $d\ge2$, and let $H^2_d$ denote the Drury--Arveson space on $\BB_d$, the open unit ball in $\CC^d$.
Fix a point $x_0\in\BB_d\setminus\{0\}$, and set $X:=\BB_d\setminus\{x_0\}$ and
$\cH:=H^2_d|_{X}$.

(i) The space $\cH$ is a RKHS on $X$ whose kernel is normalized at $0$ and has the complete Pick property.

(ii) Each function $f\in\cH$ has a unique extension to a holomorphic function on $\BB_d$.
Evaluating this extension at $x_0$ defines a continuous linear functional $\Lambda$ on $\cH$ such that $\Lambda(1)=1$.
Moreover, since holomorphic functions in two or more variables have no isolated zeros, it follows 
that $\Lambda(f)\ne0$ for every function $f\in\cH$ that does not vanish on $X$.

(iii) By considering its action on the coordinate functions, we see that $\Lambda$ is not given by evaluation at any point of $X$.
\end{proof}

The counterexamples exhibited above are closely 
related to the notion of \emph{algebraic consistency} or \emph{maximal domains}.
To elaborate somewhat, a non-zero continuous linear functional $\Lambda$ on a RKHS $\cH$ on $X$
is said to be \emph{partially multiplicative} if $\Lambda(fg) = \Lambda(f) \Lambda(g)$ for all $f, g \in \cH$
satisfying $f g \in \cH$. Clearly, point evaluation at any point in $X$ defines a partially
multiplicative functional. The set $X$ is said to be a \emph{maximal domain} for $\cH$
if every partially multiplicative functional is of this type.

It was shown in \cite[\S 2]{MS17} that $X$ can always be enlarged to a maximal domain
by regarding $\cH$ as a space of functions on its set of partially multiplicative functionals.
For concrete spaces of holomorphic functions on domains in $\CC^d$, one can sometimes determine
the maximal domain by considering all points $\lambda \in \CC^d$ such that the functional
$p \mapsto p(\lambda)$, defined for all polynomials $p$, is continuous with respect to the norm of $\cH$;
see \cite[\S 3]{Sa20}.
Moreover, in the case when $\cH$ has a normalized complete Pick kernel, there are alternative constructions
of the maximal domain, using weak-$*$ continuous characters on the multiplier algebra $\cM$,
Gleason parts of the character space of $\cM$, or a Zariski type closure; see Proposition 4.2
and Remark 4.3 in \cite{AHMR17}.
For more information on maximal domains, we refer the reader to  \cite{AHMR17}, \cite[Def.1.5]{CM95}, \cite[\S5]{Ha17}, \cite[\S2]{MS17}, 
and \cite[\S3]{Sa20}.

\section*{Acknowledgements}
We thank the anonymous referees for their careful reading of the paper and
for several helpful suggestions of improvements.

Chu was supported by a CRM--Laval Postdoctoral Fellowship.
Hartz was partially supported by a GIF grant.
Mashreghi was partially supported by an NSERC Discovery  Grant. 
Ransford was supported by grants from NSERC and the Canada Research Chairs program.

\bibliographystyle{plain}
\bibliography{biblist}

\end{document}